\documentclass{amsart}
\usepackage{amssymb}
\usepackage{amsmath}
\usepackage{amsfonts}

\newtheorem{theorem}{Theorem}
\theoremstyle{plain}

\newtheorem{definition}{Definition}

\newtheorem{remark}{Remark}

\numberwithin{equation}{section}
\input{tcilatex}

\begin{document}
\title[FRACTIONAL INEQUALITIES]{FRACTIONAL INTEGRAL INEQUALITIES FOR
DIFFERENT FUNCTIONS}
\author{M.EM\.{I}N \"{O}ZDEM\.{I}R$^{\blacktriangle }$}
\address{$^{\blacktriangle }$ATAT\"{U}RK UNIVERSITY, K.K. EDUCATION FACULTY,
DEPARTMENT OF MATHEMATICS, 25240, CAMPUS, ERZURUM, TURKEY}
\email{emos@atauni.edu.tr}
\author{\c{C}ET\.{I}N YILDIZ$^{\blacktriangle ,\bigstar }$}
\email{yildizc@atauni.edu.tr}
\author{HAVVA KAVURMACI$^{\blacktriangle }$}
\email{hkavurmaci@atauni.edu.tr}
\thanks{$^{\bigstar }$Corresponding Author.}

\begin{abstract}
In this paper, we establish several inequalities for different$\ $convex
mappings that are connected with the Riemann-Liouville fractional integrals.
Our results have some relationships with certain integral inequalities in
the literature.
\end{abstract}

\maketitle

\section{INTRODUCTION}

Let $f:I\subseteq 
\mathbb{R}
\rightarrow 
\mathbb{R}
$ be a convex function and let $a,b\in I,$ with $a<b.$ The following
inequality;%
\begin{equation}
\ \ \ f\left( \frac{a+b}{2}\right) \leq \frac{1}{b-a}\int_{a}^{b}f(x)dx\leq 
\frac{f(a)+f(b)}{2}  \label{a}
\end{equation}%
is known in the literature as Hadamard's inequality. Both inequalities hold
in the reversed direction if $f$ is concave.

In \cite{GL}, Godunova and Levin introduced the following class of functions.

\begin{definition}
A function $f:I\subseteq 
\mathbb{R}
\rightarrow 
\mathbb{R}
$ is said to belong to the class of $Q(I)$ if it is nonnegative and for all $%
x,y\in I$ and $\lambda \in (0,1)$ satisfies the inequality;%
\begin{equation}
f(\lambda x+(1-\lambda )y)\leq \frac{f(x)}{\lambda }+\frac{f(y)}{1-\lambda }.
\label{2}
\end{equation}
\end{definition}

They also noted that all nonnegative monotonic and nonnegative convex
functions belong to this class and also proved the following motivating
result:

If $f\in Q(I)$ and $x,y,z\in I,$ then%
\begin{equation}
f(x)(x-y)(x-z)+f(y)(y-x)(y-z)+f(z)(z-x)(z-y)\geq 0.  \label{1}
\end{equation}%
In fact (\ref{1}) is even equivalent to (\ref{2}). So it can alternatively
be used in the definition of the class $Q(I).$

In \cite{DPP}, Dragomir et.al., defined the following new class of functions.

\begin{definition}
A function $f:I\subseteq 
\mathbb{R}
\rightarrow 
\mathbb{R}
$ is $P$ function or that $f$ belongs to the class of $P(I),$ if it is
nonnegative and for all $x,y\in I$ and $\lambda \in \lbrack 0,1],$ satisfies
the following inequality;%
\begin{equation}
f(\lambda x+(1-\lambda )y)\leq f(x)+f(y).  \label{6}
\end{equation}%
The power mean $M_{r}(x,y;\lambda )$ of order $r$ of positive numbers $x,y$
is defined by%
\begin{equation*}
M_{r}(x,y;\lambda )=\left\{ 
\begin{array}{cc}
\left( \lambda x^{r}+\left( 1-\lambda \right) y^{r}\right) ^{\frac{1}{r}}, & 
r\neq 0 \\ 
x^{\lambda }y^{1-\lambda }, & r=0.%
\end{array}%
\right. 
\end{equation*}
\end{definition}

In \cite{SIMIC}, Pearce et al. generalized this inequality to $r-$convex
positive function $f$ which is defined on an interval $[a,b],$ for all $%
x,y\in \lbrack a,b]$ and $\lambda \in \lbrack 0,1];$%
\begin{equation*}
f(\lambda x+(1-\lambda )y)\leq M_{r}(f\left( x\right) ,f\left( y\right)
;\lambda )=\left\{ 
\begin{array}{cc}
\left( \lambda \left[ f(x)\right] ^{r}+(1-\lambda )\left[ f(y)\right]
^{r}\right) ^{\frac{1}{r}}, & if\text{ }r\neq 0 \\ 
\text{ }\left[ f(x)\right] ^{\lambda }\left[ f(y)\right] ^{1-\lambda } & 
\text{\ }if\text{ }r=0%
\end{array}%
\right. .
\end{equation*}%
We have that $0-$convex functions are simply log-convex functions and $1-$%
convex functions are ordinary convex functions.

In \cite{varr}, Varo\v{s}anec introduced the following class of functions.

\begin{definition}
Let $h:J\subset 
\mathbb{R}
\rightarrow 
\mathbb{R}
$ be a positive function. We say that $f:I\subset 
\mathbb{R}
\rightarrow 
\mathbb{R}
$ is $h-$convex function or that $f$ belongs to the class $SX(h,I)$, if $f$
is nonnegative and for all $x,y\in I$ and $\lambda \in (0,1)$ we have%
\begin{equation}
f(\lambda x+(1-\lambda )y)\leq h(\lambda )f(x)+h(1-\lambda )f(y).  \label{9}
\end{equation}
\end{definition}

If the inequality in (\ref{9}) is reversed, then $f$ is said to be $h-$%
concave, i.e., $f\in SV(h,I).$

Obviously, if $h\left( \lambda \right) =\lambda $, then all nonnegative
convex functions belong to $SX\left( h,I\right) $\ and all nonnegative
concave functions belong to $SV(h,I);$ if $h(\lambda )=\frac{1}{\lambda },$
then $SX(h,I)=Q(I);$ if $h(\lambda )=1,$ then $SX(h,I)\supseteq P(I)$ and if 
$h(\lambda )=\lambda ^{s},$ where $s\in \left( 0,1\right) ,$ then $%
SX(h,I)\supseteq K_{s}^{2}.$ For some recent results for $h-$convex
functions we refer to the interested reader to the papers \cite{SA2}, \cite%
{OZ1} and \cite{BU}.

In \cite{DPP}, Dragomir et.al. proved two inequalities of Hadamard type for
class of Godunova-Levin functions and $P-$ functions.

\begin{theorem}
Let $f\in Q(I),$ $a,b\in I$ with $a<b$ and $f\in L_{1}[a,b].$ Then the
following inequality holds:%
\begin{equation}
\ \ \ f\left( \frac{a+b}{2}\right) \leq \frac{4}{b-a}\int_{a}^{b}f(x)dx.
\label{a.1.1}
\end{equation}
\end{theorem}

\begin{theorem}
Let $f\in P(I),$ $a,b\in I$ with $a<b$ and $f\in L_{1}[a,b].$ Then the
following inequality holds:%
\begin{equation}
f\left( \frac{a+b}{2}\right) \leq \frac{2}{b-a}\int_{a}^{b}f(x)dx\leq
2[f(a)+f(b)].  \label{a.1.2}
\end{equation}
\end{theorem}

In \cite{PTT}, Ngoc et al., established following theorem for $r-$convex
functions:

\begin{theorem}
Let $f:[a,b]\rightarrow (0,\infty )$ be $r-$convex function on $[a,b]$ with $%
a<b.$Then the following inequality holds for $0<r\leq 1$:%
\begin{equation}
\frac{1}{b-a}\dint\limits_{a}^{b}f(x)dx\leq \left( \frac{r}{r+1}\right) ^{%
\frac{1}{r}}\left( f^{r}(a)+f^{r}(b)\right) ^{\frac{1}{r}}.  \label{1.1}
\end{equation}
\end{theorem}

For related results on $r-$convexity see the papers \cite{GS} and \cite{GILL}%
.

In \cite{zeki}, Sar\i kaya et al. proved the following Hadamard type
inequalities for $h-$convex functions.

\begin{theorem}
Let $f\in SX(h,I),$ $a,b\in I$ with $a<b$ and $f\in L_{1}[a,b].$ Then%
\begin{equation}
\frac{1}{2h\left( \frac{1}{2}\right) }f\left( \frac{a+b}{2}\right) \leq 
\frac{1}{b-a}\int_{a}^{b}f(x)dx\leq \lbrack f(a)+f(b)]\int_{0}^{1}h(\alpha
)d\alpha .  \label{10}
\end{equation}
\end{theorem}

In \cite{zeki2}, Sar\i kaya et al. proved the following Hadamard type
inequalities for fractional integrals as follows.

\begin{theorem}
\label{diz} Let $f:[a,b]\rightarrow 
\mathbb{R}
$ be positive function with $0\leq a<b$ and $f\in L_{1}[a,b].$ If $f$ is
convex function on $[a,b]$, then the following inequalities for fractional
integrals hold:%
\begin{equation}
f\left( \frac{a+b}{2}\right) \leq \frac{\Gamma (\alpha +1)}{2(b-a)^{\alpha }}%
\left[ J_{a^{+}}^{\alpha }(b)+J_{b^{-}}^{\alpha }(a)\right] \leq \frac{%
f(a)+f(b)}{2}  \label{16}
\end{equation}%
with $\alpha >0.$
\end{theorem}

Now we give some necessary definitions and mathematical preliminaries of
fractional calculus theory which are used throughout this paper.

\begin{definition}
Let $f\in L_{1}[a,b].$ The Riemann-Liouville integrals $J_{a^{+}}^{\alpha }f$
and $J_{b^{-}}^{\alpha }f$ of order $\alpha >0$ with $a\geq 0$ are defined by%
\begin{equation*}
J_{a^{+}}^{\alpha }f(x)=\frac{1}{\Gamma (\alpha )}\underset{a}{\overset{x}{%
\int }}\left( x-t\right) ^{\alpha -1}f(t)dt,\text{ \ }x>a
\end{equation*}%
and%
\begin{equation*}
J_{b^{-}}^{\alpha }f(x)=\frac{1}{\Gamma (\alpha )}\underset{x}{\overset{b}{%
\int }}\left( t-x\right) ^{\alpha -1}f(t)dt,\text{ \ }x<b
\end{equation*}%
respectively where $\Gamma (\alpha )=\underset{0}{\overset{\infty }{\int }}%
e^{-u}u^{\alpha -1}du.$ Here is $J_{a^{+}}^{0}f(x)=J_{b^{-}}^{0}f(x)=f(x).$
\end{definition}

In the case of $\alpha =1$, the fractional integral reduces to the classical
integral.

For some recent results connected with \ fractional integral inequalities
see \cite{anastas}-\cite{dahtab} and \cite{zeki2}.

The main purpose of this paper is to present new Hadamard's inequalities for
fractional integrals via functions that belongs to the classes of $Q(I),$ $%
P(I),$ $SX(h,I)$ and $r-$convex.

\section{MAIN RESULTS}

\begin{theorem}
Let $f\in Q(I),$ $a,b\in I$ with $0\leq a<b$ and $f\in L_{1}[a,b].$ Then the
following inequality for fractional integrals hold:%
\begin{equation}
\ f\left( \frac{a+b}{2}\right) \leq \frac{2\Gamma (\alpha +1)}{(b-a)^{\alpha
}}\left[ J_{a^{+}}^{\alpha }(b)+J_{b^{-}}^{\alpha }(a)\right]  \label{4}
\end{equation}
with $\alpha >0.$
\end{theorem}

\begin{proof}
Since $f\in Q(I),$ we have%
\begin{equation*}
2\left( f(x)+f(y)\right) \geq f\left( \frac{x+y}{2}\right)
\end{equation*}%
for all $x,y\in I$ (with $\lambda =\frac{1}{2}$ in (1.2)).

If we choose $x=ta+(1-t)b$ and $y=(1-t)a+tb$ in above inequality, we get 
\begin{equation}
2\left[ f(ta+(1-t)b)+f((1-t)a+tb)\right] \geq f\left( \frac{a+b}{2}\right) .
\label{3}
\end{equation}%
Then multiplying both sides of (\ref{3}) by $t^{\alpha -1}$ and integrating
the resulting inequality with respect to $t$ over $[0,1]$, we obtain%
\begin{eqnarray*}
2\int_{0}^{1}t^{\alpha -1}\left[ f(ta+(1-t)b)+f((1-t)a+tb)\right] dt &\geq
&f\left( \frac{a+b}{2}\right) \int_{0}^{1}t^{\alpha -1}dt \\
2\int_{a}^{b}\left( \frac{b-u}{b-a}\right) ^{\alpha -1}f(u)\frac{du}{b-a}%
+2\int_{a}^{b}\left( \frac{v-a}{b-a}\right) ^{\alpha -1}f(v)\frac{dv}{b-a}
&\geq &\frac{1}{\alpha }f\left( \frac{a+b}{2}\right)  \\
\frac{2\Gamma (\alpha +1)}{(b-a)^{\alpha }}\left[ J_{a^{+}}^{\alpha
}(b)+J_{b^{-}}^{\alpha }(a)\right]  &\geq &f\left( \frac{a+b}{2}\right) .
\end{eqnarray*}%
The proof is complete.
\end{proof}

\begin{remark}
If we choose $\alpha =1$ in Theorem 6, then the inequalities (\ref{4})
become the inequalities (\ref{a.1.1}).
\end{remark}

\begin{theorem}
\label{cettt} Let $f\in P(I),$ $a,b\in I$ with $a<b$ and $f\in L_{1}[a,b].$
Then one has inequality for fractional integrals:%
\begin{equation}
f\left( \frac{a+b}{2}\right) \leq \frac{\Gamma (\alpha +1)}{(b-a)^{\alpha }}%
\left[ J_{a^{+}}^{\alpha }(b)+J_{b^{-}}^{\alpha }(a)\right] \leq 2\left(
f(a)+f(b)\right)  \label{5}
\end{equation}%
with $\alpha >0.$
\end{theorem}

\begin{proof}
According to (\ref{6}) with $x=ta+(1-t)b,$ $y=(1-t)a+tb$ and $\lambda =\frac{%
1}{2},$ we find that%
\begin{equation}
f\left( \frac{a+b}{2}\right) \leq f(ta+(1-t)b)+f((1-t)a+tb)  \label{7}
\end{equation}%
for all $t\in \lbrack 0,1].$ Thus multiplying both sides of (\ref{7}) by $%
t^{\alpha -1}$ and integrating the resulting inequality with respect to $t$
over $[0,1]$, we have%
\begin{eqnarray*}
f\left( \frac{a+b}{2}\right) \int_{0}^{1}t^{\alpha -1}dt &\leq
&\int_{0}^{1}t^{\alpha -1}\left[ f(ta+(1-t)b)+f((1-t)a+tb)\right] dt \\
\frac{1}{\alpha }f\left( \frac{a+b}{2}\right)  &\leq &\frac{\Gamma (\alpha )%
}{(b-a)^{\alpha }}\left[ J_{a^{+}}^{\alpha }(b)+J_{b^{-}}^{\alpha }(a)\right]
\\
f\left( \frac{a+b}{2}\right)  &\leq &\frac{\Gamma (\alpha +1)}{(b-a)^{\alpha
}}\left[ J_{a^{+}}^{\alpha }(b)+J_{b^{-}}^{\alpha }(a)\right] 
\end{eqnarray*}%
and the first inequality is proved.

Since $f\in P(I)$, we have%
\begin{equation*}
f(ta+(1-t)b)\leq f(a)+f(b)
\end{equation*}%
and%
\begin{equation*}
f((1-t)a+tb)\leq f(a)+f(b).
\end{equation*}%
By adding these inequalities we get%
\begin{equation}
f(ta+(1-t)b)+f((1-t)a+tb)\leq 2\left[ f(a)+f(b)\right] .  \label{8}
\end{equation}%
Then multiplying both sides of (\ref{8}) by $t^{\alpha -1}$ and integrating
the resulting inequality with respect to $t$ over $[0,1]$, we have%
\begin{eqnarray*}
\int_{0}^{1}t^{\alpha -1}\left[ f(ta+(1-t)b)+f((1-t)a+tb)\right] dt &\leq &2%
\left[ f(a)+f(b)\right] \int_{0}^{1}t^{\alpha -1}dt \\
\frac{\Gamma (\alpha +1)}{(b-a)^{\alpha }}\left[ J_{a^{+}}^{\alpha
}(b)+J_{b^{-}}^{\alpha }(a)\right]  &\leq &2\left( f(a)+f(b)\right) 
\end{eqnarray*}%
and thus the second inequality is proved.
\end{proof}

\begin{remark}
If we choose $\alpha =1$ in Theorem \ref{cettt}, then the inequalities (\ref%
{5}) become the inequalities (\ref{a.1.2}).
\end{remark}

\begin{theorem}
\label{tinn}Let $f:[a,b]\rightarrow (0,\infty )$ be $r-$convex function on $%
[a,b]$ with $a<b$ and $0<r\leq 1.$Then the following inequality for
fractional integral inequlities holds$:$%
\begin{eqnarray*}
\frac{\Gamma (\alpha +1)}{(b-a)^{\alpha }}\left[ J_{a^{+}}^{\alpha
}(b)+J_{b^{-}}^{\alpha }(a)\right] &\leq &\left[ \left( \frac{1}{\alpha +%
\frac{1}{r}}\right) ^{r}\left[ f(a)\right] ^{r}+\left( \beta (\alpha ,\frac{%
r+1}{r})\right) ^{r}\left[ f(b)\right] ^{r}\right] ^{\frac{1}{r}} \\
&&+\left[ \left( \beta (\alpha ,\frac{r+1}{r})\right) ^{r}\left[ f(a)\right]
^{r}+\left( \frac{1}{\alpha +\frac{1}{r}}\right) ^{r}\left[ f(b)\right] ^{r}%
\right] ^{\frac{1}{r}}.
\end{eqnarray*}
\end{theorem}

\begin{proof}
Since $f$ is $r$-convex function and $r>0$, we have%
\begin{equation*}
f(ta+(1-t)b)\leq \left( t\left[ f(a)\right] ^{r}+(1-t)\left[ f(b)\right]
^{r}\right) ^{\frac{1}{r}}
\end{equation*}%
and%
\begin{equation*}
f((1-t)a+tb)\leq \left( (1-t)\left[ f(a)\right] ^{r}+t\left[ f(b)\right]
^{r}\right) ^{\frac{1}{r}}
\end{equation*}%
for all $t\in \lbrack 0,1].$

By adding these inequalities we have%
\begin{equation*}
f(ta+(1-t)b)+f((1-t)a+tb)\leq \left( t\left[ f(a)\right] ^{r}+(1-t)\left[
f(b)\right] ^{r}\right) ^{\frac{1}{r}}+\left( (1-t)\left[ f(a)\right] ^{r}+t%
\left[ f(b)\right] ^{r}\right) ^{\frac{1}{r}}.
\end{equation*}%
Then multiplying both sides of above inequality by $t^{\alpha -1}$ and
integrating the resulting inequality with respect to $t$ over $[0,1]$, we
obtain%
\begin{eqnarray*}
&&\int_{0}^{1}t^{\alpha -1}\left[ f(ta+(1-t)b)+f((1-t)a+tb)\right] dt \\
&\leq &\int_{0}^{1}t^{\alpha -1}\left( t\left[ f(a)\right] ^{r}+(1-t)\left[
f(b)\right] ^{r}\right) ^{\frac{1}{r}}dt+\int_{0}^{1}t^{\alpha -1}\left(
(1-t)\left[ f(a)\right] ^{r}+t\left[ f(b)\right] ^{r}\right) ^{\frac{1}{r}%
}dt.
\end{eqnarray*}%
It is easy to observe that%
\begin{equation*}
\int_{0}^{1}t^{\alpha -1}\left[ f(ta+(1-t)b)+f((1-t)a+tb)\right] dt=\frac{%
\Gamma (\alpha )}{(b-a)^{\alpha }}\left[ J_{a^{+}}^{\alpha
}(b)+J_{b^{-}}^{\alpha }(a)\right]
\end{equation*}%
Using Minkowski's inequality, we have%
\begin{eqnarray*}
\int_{0}^{1}t^{\alpha -1}\left( t\left[ f(a)\right] ^{r}+(1-t)\left[ f(b)%
\right] ^{r}\right) ^{\frac{1}{r}}dt &\leq &\left[ \left(
\int_{0}^{1}t^{\alpha +\frac{1}{r}-1}f(a)dt\right) ^{r}+\left(
\int_{0}^{1}t^{\alpha -1}(1-t)^{\frac{1}{r}}f(b)dt\right) ^{r}\right] ^{%
\frac{1}{r}} \\
&=&\left[ \left( \frac{1}{\alpha +\frac{1}{r}}\right) ^{r}\left[ f(a)\right]
^{r}+\left( \beta (\alpha ,\frac{r+1}{r})\right) ^{r}\left[ f(b)\right] ^{r}%
\right] ^{\frac{1}{r}}
\end{eqnarray*}%
and similarly%
\begin{eqnarray*}
\int_{0}^{1}t^{\alpha -1}\left( (1-t)\left[ f(a)\right] ^{r}+t\left[ f(b)%
\right] ^{r}\right) ^{\frac{1}{r}} &\leq &\left[ \left(
\int_{0}^{1}t^{\alpha -1}(1-t)^{\frac{1}{r}}f(a)dt\right) ^{r}+\left(
\int_{0}^{1}t^{\alpha +\frac{1}{r}-1}f(b)dt\right) ^{r}\right] ^{\frac{1}{r}}
\\
&=&\left[ \left( \beta (\alpha ,\frac{r+1}{r})\right) ^{r}\left[ f(a)\right]
^{r}+\left( \frac{1}{\alpha +\frac{1}{r}}\right) ^{r}\left[ f(b)\right] ^{r}%
\right] ^{\frac{1}{r}}.
\end{eqnarray*}%
Thus%
\begin{eqnarray*}
\frac{\Gamma (\alpha +1)}{(b-a)^{\alpha }}\left[ J_{a^{+}}^{\alpha
}(b)+J_{b^{-}}^{\alpha }(a)\right] &\leq &\left[ \left( \frac{1}{\alpha +%
\frac{1}{r}}\right) ^{r}\left[ f(a)\right] ^{r}+\left( \beta (\alpha ,\frac{%
r+1}{r})\right) ^{r}\left[ f(b)\right] ^{r}\right] ^{\frac{1}{r}} \\
&&+\left[ \left( \beta (\alpha ,\frac{r+1}{r})\right) ^{r}\left[ f(a)\right]
^{r}+\left( \frac{1}{\alpha +\frac{1}{r}}\right) ^{r}\left[ f(b)\right] ^{r}%
\right] ^{\frac{1}{r}}.
\end{eqnarray*}%
This proof is complete.
\end{proof}

\begin{remark}
\bigskip In Theorem \ref{tinn}, if we choose $\alpha =1,$ then we obtain the
inequalities (\ref{1.1}).
\end{remark}

\begin{theorem}
\label{yil} Let $f\in SX(h,I),$ $a,b\in I$ with $a<b$ and $f\in L_{1}[a,b].$
Then one has inequality for $h-$convex functions via fractional integrals%
\begin{eqnarray}
\frac{1}{\alpha h\left( \frac{1}{2}\right) }f\left( \frac{a+b}{2}\right)
&\leq &\frac{\Gamma (\alpha )}{(b-a)^{\alpha }}\left[ J_{a^{+}}^{\alpha
}(b)+J_{b^{-}}^{\alpha }(a)\right]  \label{11} \\
&\leq &\left[ f(a)+f(b)\right] \int_{0}^{1}t^{\alpha -1}\left[ h(t)+h(1-t)%
\right] dt.  \notag
\end{eqnarray}
\end{theorem}

\begin{proof}
According to (\ref{9}) with $x=ta+(1-t)b$, $y=(1-t)a+tb$ and $\alpha =\frac{1%
}{2}$ we find that%
\begin{eqnarray}
f\left( \frac{a+b}{2}\right)  &\leq &h\left( \frac{1}{2}\right)
f(ta+(1-t)b)+h\left( \frac{1}{2}\right) f((1-t)a+tb)  \label{h} \\
&\leq &h\left( \frac{1}{2}\right) \left[ f(ta+(1-t)b)+f((1-t)a+tb)\right] . 
\notag
\end{eqnarray}%
Then multiplying the firts inequalitiy in (\ref{h}) by $t^{\alpha -1}$ and
integrating the resulting inequality with respect to $t$ over $[0,1]$, we
obtain%
\begin{eqnarray}
f\left( \frac{a+b}{2}\right) \int_{0}^{1}t^{\alpha -1}dt &\leq &h\left( 
\frac{1}{2}\right) \int_{0}^{1}t^{\alpha -1}\left[ f(ta+(1-t)b)+f((1-t)a+tb)%
\right] dt  \notag \\
\frac{1}{\alpha h\left( \frac{1}{2}\right) }f\left( \frac{a+b}{2}\right) 
&\leq &\frac{\Gamma (\alpha )}{(b-a)^{\alpha }}\left[ J_{a^{+}}^{\alpha
}(b)+J_{b^{-}}^{\alpha }(a)\right]   \label{18}
\end{eqnarray}%
and the first inequality in (\ref{11}) is proved.

Since $f\in SX(h,I)$, we have%
\begin{equation*}
f(tx+(1-t)y)\leq h(t)f(x)+h(1-t)f(y)
\end{equation*}%
and%
\begin{equation*}
f((1-t)x+ty)\leq h(1-t)f(x)+h(t)f(y).
\end{equation*}%
By adding these inequalities we get%
\begin{equation}
f(tx+(1-t)y)+f((1-t)x+ty)\leq \left[ h(t)+h(1-t)\right] \left[ f(x)+f(y)%
\right] .  \label{12}
\end{equation}

By using (\ref{12}) with $x=a$ and $y=b$ we have%
\begin{equation}
f(ta+(1-t)b)+f((1-t)a+tb)\leq \left[ h(t)+h(1-t)\right] \left[ f(a)+f(b)%
\right] .  \label{13}
\end{equation}%
Then multiplying both sides of (\ref{13}) by $t^{\alpha -1}$ and integrating
the resulting inequality with respect to $t$ over $[0,1]$, we get%
\begin{eqnarray*}
&&\int_{0}^{1}t^{\alpha -1}\left[ f(ta+(1-t)b)+f((1-t)a+tb)\right] dt \\
&\leq &\int_{0}^{1}t^{\alpha -1}\left[ h(t)+h(1-t)\right] \left[ f(a)+f(b)%
\right] dt,
\end{eqnarray*}%
\begin{eqnarray}
&&\frac{\Gamma (\alpha )}{(b-a)^{\alpha }}\left[ J_{a^{+}}^{\alpha
}(b)+J_{b^{-}}^{\alpha }(a)\right]   \label{14} \\
&\leq &\left[ f(a)+f(b)\right] \int_{0}^{1}t^{\alpha -1}\left[ h(t)+h(1-t)%
\right] dt  \notag
\end{eqnarray}%
and thus the second inequality is proved. We obtain inequalities (\ref{11})
from (\ref{18}) and (\ref{14}).

The proof is complete.
\end{proof}

\begin{remark}
\end{remark}

\begin{itemize}
\item If we choose $h(t)=t$ in Theorem \ref{yil}, then the inequalities (\ref%
{11}) become the inequalities (\ref{16}) of Theorem \ref{diz}.

\item In Theorem \ref{yil}, if we take $\alpha =1,$ then we obtain the
inequalities (\ref{10}).

\item Let $\alpha =1.$ In Theorem \ref{yil}, if we choose $h(t)=t$ and $%
h(t)=1,$ then (\ref{11}) reduce to (\ref{a}) and (\ref{a.1.2}), respectively.
\end{itemize}

\end{document}